\documentclass{article}

\usepackage{amsfonts}
\usepackage{amsmath}
\usepackage{amssymb}
\usepackage{wasysym}
\usepackage{yfonts}  
\usepackage{amsthm}
\usepackage{amscd}   
\usepackage[all]{xypic}
\usepackage{mathrsfs}
  
\CompileMatrices

\newtheorem{theorem}{Theorem}[section]
\newtheorem{lemma}[theorem]{Lemma}
\newtheorem{proposition}[theorem]{Proposition}
\newtheorem{corollary}[theorem]{Corollary} 
\newtheorem*{definition}{Definition}

\newtheorem*{notation}{Semantic conventions}
\newtheorem*{acknowledgement}{Acknowledgements}
\newtheorem*{thm}{Theorem}

\numberwithin{equation}{section}

\begin{document}
\title{Ascending HNN extensions of polycyclic groups have the same cohomology as their profinite completions}
\author{
{\sc Karl Lorensen}\\
\\
Mathematics Department\\
Pennsylvania State University, Altoona College\\
Altoona, PA 16601-3760\\
USA\\
e-mail: {\tt kql3@psu.edu}
}

\maketitle

\begin{abstract} Assume $G$ is a polycyclic group and $\phi:G\to G$ an endomorphism.
Let $G\ast_{\phi}$ be the ascending HNN extension of $G$ with respect to $\phi$; that is, $G\ast_{\phi}$ is given by the presentation
$$G\ast_{\phi}=\langle G, t \ |\  t^{-1}gt = \phi(g)\ \mbox{for all}\ g\in G\rangle.$$
Furthermore, let $\widehat{G\ast_{\phi}}$ be the profinite completion of $G\ast_{\phi}$.
We prove that, for any finite, discrete $\widehat{G\ast_{\phi}}$-module $A$, the map $H^*(\widehat{G\ast_{\phi}}, A)\to H^*(G\ast_{\phi},A)$ induced by the canonical map $G\ast_{\phi}\to \widehat{G\ast_{\phi}}$ is an isomorphism.

\vspace{12pt}

\end{abstract}

\section{Introduction}

\indent If $\phi:G\to G$ is a group monomorphism, the {\it ascending HNN extension} of $G$ with respect to $\phi$, denoted $G\ast_{\phi}$, is defined by
$$G\ast_{\phi}=\langle G, t \ |\  t^{-1}gt = \phi(g)\ \mbox{for all}\ g\in G\rangle.$$
This paper is concerned with ascending HNN extensions of polycyclic groups. These types of 
ascending HNN extensions merit study for an important reason: they comprise precisely those finitely generated solvable groups whose finitely generated subgroups are all finitely presented. This characterization is established in \cite{groves}, where the structure of these groups is examined in detail. Another salient property manifested by every ascending HNN extension of a polycyclic group is that of residual finiteness, proved in \cite{wise}. 

Our goal is to prove the following result about the profinite completion $\widehat{G\ast_{\phi}}$ of
$G\ast_{\phi}$ if $G$ is polycyclic.

\begin{thm} Let $G$ be a polycyclic group and $\phi:G\to G$ a monomorphism. Then, for any finite, discrete $\widehat{G\ast_{\phi}}$-module $A$, the map 
\begin{equation} H^n(\widehat{G\ast_{\phi}}, A)\to H^n(G\ast_{\phi},A)\end{equation} induced by the canonical map $G\ast_{\phi}\to \widehat{G\ast_{\phi}}$ is an isomorphism for all $n\geq 0$. 
\end{thm}

\noindent Such groups whose cohomology coincides with that of their profinite completions for all finite coefficient modules were described as ``good groups" by J-P. Serre in [{\bf 14}, Exercise 2, Chapter 2], an appellation that has persisted to this day. These groups have sparked a great deal of interest  recently, partly due to their applications to geometry; see, for instance, \cite{bianchi}, \cite{lorensen}, and \cite{schroeer}. As reflected by the examples in these three references, the property of ``goodness"
 often accompanies strong forms of residual finiteness, like subgroup separability or cyclic subgroup separability.  

Among the most elementary examples of ``good groups" are free groups and polycyclic groups.  
On the surface, ascending HNN extensions of these types of groups would appear to be unlikely candidates for Serre's cohomological property, since, although residually finite, they are not, in general,  cyclic subgroup separable. However, in \cite{def} it was established that, if $G$ is a finitely generated free group and $\phi: G\to G$ a monomorphism, then $G\ast_{\phi}$ is ``good." Moreover, strong evidence that this might also be true if $G$ is polycyclic was adduced in \cite{lorensen}, where the author showed that, in this case, the map (1.1) is an isomorphism for $n=2$. 
 
Our proof of the above theorem, presented in Section 3, begins with the observation that, for an arbitrary group $G$, $G\ast_{\phi}=G_{\phi}\rtimes \mathbb Z$, where $G_{\phi}$ is the direct limit of the sequence
$$G\stackrel{\phi}{\to} G\stackrel{\phi}{\to} G\stackrel{\phi}{\to} \cdots.$$ 
Furthermore, if $G$ is finitely generated, then $\widehat{G\ast_{\phi}}=\widehat{G_{\phi}}\rtimes \hat{\mathbb Z}$, where $\widehat{G_{\phi}}$ is the profinite completion of $G_{\phi}$.  From the Mayer-Vietoris sequences arising from these two semidirect product decompositions, we can see that, in order to prove that $G\ast_{\phi}$ is ``good," it suffices to show that $G_{\phi}$ is ``good." We prove the latter assertion for $G$ polycyclic by induction on the solvability length of $G$. Of pivotal importance for both the base case and the inductive step is that $G$'s being finitely generated ensures that $G_{\phi}$ is finitely generated as a topological group with respect to its profinite topology.  In fact, with the aid of this  observation, the inductive step can be accomplished by a very routine argument. The base case, where $G$ is abelian, on the other hand, presents considerably more difficulty, requiring a shift of focus from cohomology to homology via the universal coefficient theorem. The key ingredient in the proof of this case turns out to be the fact that the profinite completion of $H_n(G_{\phi},\mathbb Z)$ is isomorphic to $H_n(\widehat{G_{\phi}},\hat{\mathbb Z})$, provided $G$ is free abelian of finite rank. In order to establish this property, we employ the relation between homology and exterior powers, which forms the subject of Section 2. 

The identification of this new class of residually finite, solvable groups that enjoy Serre's cohomological property invites the question whether there are perhaps other varieties of ``good" solvable groups, still awaiting discovery. One natural class to consider next is that of residually finite, solvable, minimax groups, which includes all the ascending HNN extensions of polycyclic groups. As our final result, we prove that every group in this larger class is indeed ``good" as well.

\begin{thm} Let $G$ be a residually finite, solvable, minimax group. Then, for any finite, discrete $\hat{G}$-module $A$, the map 
\begin{equation} H^n(\hat{G}, A)\to H^n(G,A)\end{equation} induced by the canonical map $G\to \hat{G}$ is an isomorphism for all $n\geq 0$. 
\end{thm}

\vspace{20pt}

\begin{notation}{\rm When using the term ``group,"  we will mean an abstract group; profinite groups will always be identified with the adjective ``profinite." 

When employing the cohomology and homology of a profinite group, we will always mean continuous cohomology and profinite homology, respectively. Moreover, the same conventions apply to the $\mbox{Ext}$ and $\mbox{Tor}$ functors.

The {\it profinite topology} on a group is the topology whose basis at the identity consists of all the normal subgroups of finite index. A group $G$ is {\it finitely generated relative to its profinite topology} if it is finitely generated as a topological group, where the topology employed is the profinite topology. This is equivalent to the assertion that there exist $g_1,\cdots, g_n\in G$ such that, for every epimorphism $\epsilon$ from $G$ onto a finite group $F$, $\epsilon(g_1),\cdots, \epsilon(g_n)$ generate $F$.

When we refer to a ``finitely generated profinite group," we will always mean finitely generated in the topological sense.

If $G$ is a group, then $\hat{G}$ denotes its profinite completion and $c_G:G\to \hat{G}$ the completion map.

A {\it graded ring} $R^\ast$ is a sequence $(R^n)_{n=0}^{\infty}$ of additive abelian groups such that        $R^0=\mathbb Z$, together with bilinear product maps $R^i\times R^j\to R^{i+j}$ for all $i, j\in \mathbb N$ that obey the associative property. If $R^\ast$ is a graded ring such that, for each odd natural number $n$, $x^2=0$ for all $x\in R^n$, then $R^\ast$ is {\it strictly anticommutative}.  

A {\it profinite graded ring} $\Omega^\ast$ is a sequence $(\Omega^n)_{n=0}^{\infty}$ of additive profinite abelian groups such that $\Omega^0=\hat{\mathbb Z}$, together with continuous,  bilinear product maps $\Omega^i\times \Omega^j\to \Omega^{i+j}$ for all $i, j\in \mathbb N$, obeying the associative property.   

For a group $G$ we denote the homology group $H_n(G,\mathbb Z)$, where the action of $G$ on $\mathbb Z$ is trivial, by $H_n(G)$. Similarly, if $\Gamma$ is a profinite group, then $H_n(\Gamma)$ represents $H_n(\Gamma, \hat{\mathbb Z})$, where the action of $\Gamma$ on $\hat{\mathbb Z}$ is trivial.

Although in the introduction we only referred to ascending HNN extensions with respect to monomorphisms, in the body of the paper we will form these constructions with respect to arbitrary endomorphisms. Hence, if $\phi:G\to G$ is a group endomorphism, 
$$G\ast_{\phi}=\langle G, t \ |\  t^{-1}gt = \phi(g)\ \mbox{for all}\ g\in G\rangle.$$ 
The notation $G\ast_{\phi}$ that we employ is borrowed from \cite{geoghegan}. }

\end{notation}

\section{Homology and exterior powers}

This section is devoted to proving the formula
 \begin{equation} \widehat{H_n(G)}\cong H_n(\hat{G})\end{equation}
 if $G$ is a torsion-free, abelian group that is finitely generated with respect to its profinite topology. This formula will play an important role in the proof of the main theorem in Section 3. 
The proof of (2.1) is based on the connection between homology and exterior powers, for both abstract and profinite abelian groups. 

If $G$ is an  abelian group, then we denote the exterior power ring of $G$ by $\bigwedge^{\ast} G$. This is a graded ring that can be represented in each positive dimension $n$ as the quotient of $\bigotimes_{i=1}^n G$ by the subgroup generated by all elements of the form $g_1\otimes \cdots \otimes g_n$ such that $g_i=g_{i+1}$ for some $i$, with the multiplication defined by extending the tensoring operation linearly. The exterior power ring, then, is strictly anticommutative and enjoys the following universal property: for any strictly anticommutative graded ring $R^{\ast}$ and group homomorphism $\theta:G\to R^1$, there exists a unique graded ring homomorphism $\phi^{\ast}:\bigwedge^{\ast} G\to R^{\ast}$ such that $\phi^1=\theta$. H. Cartan \cite{cartan} established the following connection between the integral homology of a finitely generated, torsion-free abelian group and its exterior power; for a more contemporary proof, in English, see [{\bf 1}, p. 123]. 

\begin{theorem} {\rm (Cartan)} If $G$ is a finitely generated, torsion-free abelian group and $H_{\ast}(G)$ is regarded as a graded ring using the Pontryagin product, then

$$H_{\ast}(G)\cong \mbox{$\bigwedge^{\ast}$} G.$$
\end{theorem}

For a profinite abelian group $\Gamma$ we represent the profinite exterior power ring of $\Gamma$ by $\hat{\bigwedge}^{\ast} \Gamma$; it is a strictly anticommutative graded profinite ring with the following universal property: for any strictly anticommutative graded profinite ring $\Omega^\ast$ and any continuous group homomorphism $\theta:G\to \Omega^1$, there exists a unique continuous homomorphism of graded rings $\phi^{\ast}:\bigwedge^{\ast} G\to \Omega^{\ast}$ such that $\phi^1=\theta$. As described in [{\bf 16}, p. 131],  $\hat{\bigwedge}^n \Gamma$ can be constructed as the completion of $\bigwedge^n \Gamma$ with respect to the kernels of all the maps $\bigwedge^n \Gamma\to \bigwedge^n \Gamma/N$ for $N\unlhd_o \Gamma$. Alternatively, it may viewed as the quotient of the completed tensor product $\hat{\bigotimes}^n \Gamma$ by the closed subgroup generated by all elements of the form $g_1\hat{\otimes} \cdots \hat{\otimes} g_n$ such that $g_i=g_{i+1}$ for some $i$.  
As established below, the profinite exterior power coincides with the abstract exterior power for a finitely generated profinite abelian group.

\begin{proposition} If $\Gamma$ is a finitely generated profinite abelian group, then the canonical map
${\bigwedge}^{\ast} \Gamma\to \hat{\bigwedge}^{\ast} \Gamma$
is an isomorphism of graded rings.
\end{proposition}

\begin{proof} We will prove by induction that ${\bigwedge}^n \Gamma\cong \hat{\bigwedge}^n \Gamma$
for every $n\in \mathbb N$. Consider the canonical  group epimomorphism 
$$\phi: \mbox{$\bigwedge$}^{n-1} \Gamma \otimes \Gamma\to \mbox{$\bigwedge$}^n\Gamma.$$
 By the inductive hypothesis, ${\bigwedge}^{n-1} \Gamma$ is a finitely generated profinite group. Hence, by [{\bf 10}, Proposition 5.5.3(d)], the domain of $\phi$ is a profinite group. Thus $\bigwedge^n \Gamma$ is compact in the profinite topology, yielding the desired result.
\end{proof}

We wish to prove a profinite analogue of Cartan's theorem. In order to do so, we require the following
K\"unneth formula. 

\begin{theorem} Let $\Gamma_1$ and $\Gamma_2$ be profinite groups. Then there is an exact sequence

\begin{equation} \begin{CD}
0 @>>> \bigoplus_{i+j=n} H_i(\Gamma_1)\hat{\otimes} H_j(\Gamma_2) @>>> H_n(\Gamma_1\times \Gamma_2) @>>>  \bigoplus_{i+j=n-1}{\rm Tor}(H_i(\Gamma_1), H_j(\Gamma_2)) @>>>
0.\end{CD} \end{equation}
\end{theorem}

\noindent Although well known, the above formula does not appear to be proven anywhere in the literature. Nevertheless, the proof of the abstract version, as presented, for example, in [{\bf 15}, Proposition 6.1.13], can be carried over with ease to the profinite realm. 

To prove our profinite version of Cartan's theorem, we also need the following property of profinite exterior powers, which may be proved in the same manner as the analogous result for abstract groups; see [{\bf 1}, p. 122].

\begin{lemma} If $\Gamma_1$ and $\Gamma_2$ are profinite abelian groups, then
$$\mbox{$\hat{\bigwedge}$}^n(\Gamma_1\oplus \Gamma_2) \cong \bigoplus_{i+j=n} \mbox{$\hat{\bigwedge}$}^i\Gamma_1 \hat{\otimes} \mbox{$\hat{\bigwedge}$}^j\Gamma_2$$
for every $n\geq 0$.
\end{lemma}

Armed with the above two results, we can readily prove the desired formula for the homology of a finitely generated profinite abelian group.

\begin{theorem}  If $\Gamma$ is a torsion-free, finitely generated profinite abelian group, then

$$H_n(\Gamma)\cong \mbox{$\bigwedge$}^n \Gamma$$
for all $n\geq 0$. 
\end{theorem}

\begin{proof} For $n=0,1$ the result is trivial; hence we assume $n\geq 2$. The profinite group $\Gamma$ can be expressed as the direct sum of finitely many infinite procyclic groups. We prove the result by induction on the number of procyclic groups in this decomposition. First assume $\Gamma$ is itself procyclic. Then $\Gamma$ is a projective profinite group, making $H_n(\Gamma)=0$ for all $n\geq 2$. In addition, $\bigwedge^n \Gamma$ is trivial if $n\geq 2$, thus confirming the result. Now assume $\Gamma=\Gamma_1 \oplus \Gamma_2$, where $\Gamma_2$ is infinite procyclic. Applying Theorem 2.3 to $\Gamma$, we have that the fourth term in sequence (2.2) is trivial. Therefore, by the inductive hypothesis together with Lemma 2.4, we have that $H_n(\Gamma)\cong \hat{\bigwedge}^n \Gamma = \bigwedge^n\Gamma$.

\end{proof} 

In order to use the above theorem to prove formula (2.1), we need to establish that 
\begin{equation} \widehat{\mbox{$\bigwedge$}^n G}\cong \mbox{$\bigwedge$}^n \hat{G}. 
\end{equation} if $G$ is an abelian group that is finitely generated relative to its profinite topology. The proof of (2.3) is based on the following lemma about graded rings. 

\begin{lemma} Let $R^{\ast}$ be a strictly anticommutative graded ring  such that $R^n$ is finitely generated in its profinite topology for each $n\in \mathbb N$. Then the following two statements hold. 
\vspace{10pt}

\noindent (i) The family $\{\widehat{R^n}\ |\ n\geq 0\}$ of profinite abelian groups  can be made into a strictly anticommutative graded profinite ring so that the completion maps $c^n:R^n\to \widehat{R^n}$ constitute a graded ring homomorphism.
\vspace{10pt}

\noindent (ii) For any strictly anticommutative graded profinite ring $\Omega^\ast$ and graded ring homomorphism $\phi^\ast: R^\ast\to \Omega^\ast$, there exists a unique continuous graded ring homomorphism $\psi^\ast:\widehat{R^\ast}\to \Omega^\ast$ such that $\psi^\ast c^\ast =\phi^\ast$.
\end{lemma}

\begin{proof}  (i). Let $f_{ij}: R^i\times R^j\to R^{i+j}$ be the function arising from the multiplication in $R^\ast$. We may obtain a function $g_{ij}:\widehat{R^i}\times R^j\to \widehat{R^{i+j}}$ extending $f_{ij}$ such that, for each $y\in R^j$, $g_{ij}(\_ , y)$ is a continuous homomorphism $\widehat{R^i}\to \widehat{R^{i+j}}$. We claim that $g_{ij}$ is continuous, where $R^j$ is given the profinite topology. 
This will follow if we can show that  $g_{ij}^{-1}(\bar{N}+a)$ is open in $\widehat{R^i}\times R^j$ for any $\bar{N}\unlhd_o \widehat{R^{i+j}}$ and $a\in \widehat{R^{i+j}}$.  For each $y\in R^j$, set
$\bar{M_y}=\{x\in \widehat{R^i}\ |\ g_{ij}(x,y)\in \bar{N}\}$. Since the map $x\mapsto \bar{N}+g_{ij}(x,y)$ from $\widehat{R^i}$ to $\widehat{R^{i+j}}/\bar{N}$ is a continuous homomorphism with kernel $\bar{M}_y$, we have that $[\widehat{R^i}:\bar{M}_y]$ divides $\widehat{[R^{i+j}}:\bar{N}]$ for every $y\in R^j$. However, $\widehat{R^i}$, being a finitely generated profinite group, possesses only finitely many open subgroups of any given index, which means that the set $\{\bar{M}_y:y\in R^j\}$ is finite. Thus $\bar{M}=\bigcap_{y\in R^j}\bar{M}_y$ is an open subgroup of $\widehat{R}^i$.  Next take $N$ to be the preimage of $\bar{N}$ under the completion map $R^{i+j}\to \widehat{R^{i+j}}$. Proceeding in a fashion similar to above,  we let $P_x=\{y\in R^j\ |\ f_{ij}(x,y)\in N\}$ for each $x\in R^i$, obtaining that $P=\bigcap_{x\in R^i}P_x$ is an open subgroup of $R^j$. 

Now assume $(b, c)\in g_{ij}^{-1}(\bar{N}+a)$. Let $b'$ be an element of the image of $R^i$ in $\widehat{R^i}$ such that $\bar{M}+b=\bar{M}+b'$. For any $m\in \bar{M}$ and $p\in P$, 
$$g_{ij}(m+b',p+c)=g_{ij}(m,p+c)+g_{ij}(b',p+c)\\
=g_{ij}(m,p+c)+g_{ij}(b',p)+g_{ij}(b',c).$$
Since $g_{ij}(m,p+c)\in \bar{N}$, $g_{ij}(b',p)\in \bar{N}$, and $g_{ij}(b',c)\in \bar{N}+a$, we have that $g_{ij}(m+b',p+c)\in \bar{N}+a$. Thus $(\bar{M}+b)\times (P+c)\leq  g_{ij}^{-1}(\bar{N}+a)$. Therefore, $g_{ij}^{-1}(\bar{N}+a)$ is open in $\widehat{R^i}\times R^j$. It follows, then, that $g_{ij}$ is continuous. As a consequence,  we can deduce that $g_{ij}$ is linear in the second component by virtue of its being linear there on a dense subset.
This allows us to extend $g_{ij}$ to a function $h_{ij}:\widehat{R^i}\times \widehat{R^j}\to \widehat{R^{i+j}}$  such that, for each $x\in \widehat{R^i}$, $h_{ij}(x, \_)$ is a continuous homomorphism $\widehat{R^j}\to \widehat{R^{i+j}}$. By reasoning like we did above for $g_{ij}$, we can conclude that $h_{ij}$ is continuous and, therefore, bilinear. The maps $h_{ij}$, then, furnish the desired product on the family of groups $\{R^n\ |\ n\in \mathbb N\}$, the associativity and strict anticommutativity following from the fact that these properties hold on dense subsets.
\vspace{5pt}

(ii). The universal property of the profinite completion yields a family $\{\psi^n:\widehat{R^n}\to \Omega^n\ |\ n\in \mathbb N\}$ of continuous group homomorphisms such that $\psi^n c^n=\phi^n$. Moreover, for all $(x, y)\in \widehat{R^i}\times \widehat{R^j}$, $\psi^{i+j}(xy)=\psi^i(x)\psi^j(y)$, since this identity holds on a dense subset. 

\end{proof}

Now we are prepared to prove formula (2.3).

\begin{proposition} Assume $G$ is an abelian group that is finitely generated with respect to its profinite topology. Then the following two statements hold. 
\vspace{10pt}
 
\noindent (i) The family of profinite abelian groups $\{\widehat{\bigwedge^n G}\ |\ n\geq 0\}$ can be made into a strictly anticommutative graded profinite ring so that the completion maps $c^n: \bigwedge^n G\to \widehat{\bigwedge^n G}$ constitute a graded ring homomorphism.
\vspace{10pt}

\noindent (ii) The graded ring homomorphism $\bigwedge^{\ast} G\to \bigwedge^{\ast} \hat{G}$ arising from $c_G:G\to \hat{G}$ induces a continuous graded profinite ring isomorphism  
$\widehat{\bigwedge^{\ast} G}\to \bigwedge^{\ast} \hat{G}.$

\end{proposition}

\begin{proof} Throughout the proof, we will make repeated use of the fact that $\bigwedge^{\ast}\hat{G}= \hat{\bigwedge}^{\ast}\hat{G}$, which follows from Proposition 2.2. Since $G$ is finitely generated relative to its profinite topology, the same is true for $\bigwedge^n G$ for $n>1$. Hence statement (i) follows by the preceding proposition. Moreover, by the universal property of $\hat{\bigwedge}^\ast \hat{G}$, there is a continuous graded ring homomorphism $\phi^{\ast}: \bigwedge^\ast \hat{G}\to \widehat{\bigwedge^\ast  G}$ such that $\phi^1$ is just the identity map $\hat{G}\to \hat{G}$. In addition, we have a graded ring homomorphism $\bigwedge^{\ast} G\to {\bigwedge}^{\ast} \hat{G}$, which, according to Lemma 2.6(ii), induces a 
continuous graded ring homomorphism $\psi^{\ast}: \widehat{\bigwedge^{\ast} G}\to {\bigwedge}^{\ast} \hat{G}$ such that $\psi^1$ is the identity map $\hat{G}\to \hat{G}$. 
 We claim that $\psi^{\ast}\phi^{\ast}$ is the identity map ${\bigwedge}^{\ast} \hat{G}\to {\bigwedge}^{\ast} \hat{G}$, and that $\phi^{\ast}\psi^{\ast}$ is the identity map $\widehat{\bigwedge^{\ast}  G}\to \widehat{\bigwedge^{\ast} G}$. The first assertion follows immediately from the universal property of $\hat{\bigwedge}^{\ast} \hat{G}$ since $\psi^1\phi^1$ is the identity map $\hat{G}\to \hat{G}$. To verify the second, consider the composition
 
 \begin{displaymath} \begin{CD} \bigwedge^\ast  G @>c^\ast>> \widehat{\bigwedge^\ast  G} @>\phi^{\ast}\psi^{\ast}>> \widehat{\bigwedge^\ast  G}.\\
\end{CD} \end{displaymath}

 \noindent The universal property of the exterior power ensures that this composition is the completion map since that is its form in dimension one. Consequently, by the universal property of the profinite completion, $\phi^{\ast}\psi^{\ast} $ can only be the identity map.
 \end{proof}

In order to deduce formula (2.1) from the above proposition, we still require the fact that the profinite completion of a torsion-free abelian group is itself torsion-free. This property may be established by the same argument used to prove [{\bf 8}, Proposition 2.1].  Notice that for our result, as opposed to the formulation in \cite{kropholler}, it is not necessary to assume that $G$ is residually finite.

\begin{lemma}{\rm (P. Kropholler and J. Wilson)} If $G$ is a torsion-free abelian group, then $\hat{G}$ is also torsion-free. 
\end{lemma}

\begin{proof}

Assume $\bar{x}\in \hat{G}$ such that $n\bar{x}=0$ for some integer $n\neq 0$. Let $N$ be a subgroup of finite index in $G$, and let $x\in G$ such that $N+x$ is the image of $\bar{x}$ in $G/N$. Then $nx\in N$. We claim that, in fact,  $nx\in nN$. To show this, suppose otherwise. As an abelian group with finite exponent, $N/nN$ is a direct sum of finite cyclic groups and thus residually finite. Hence $N$ contains a subgroup $M$ of finite index such that $nN\leq M$ and $nx\notin M$. Now let $y\in G$ such that $M+y$ is the image of $\bar{x}$ in $G/M$. Hence $ny\in M$ and $x-y\in N$. 
 It follows that $n(x-y)\in nN\leq M$, so that $nx\in M$, a contradiction. Therefore, $nx\in nN$, which means, since $G$ is torsion-free, that $x\in N$. Since $N$ was an
arbitrary subgroup of finite index in $G$, we can conclude that $\bar{x}=0$.  Consequently, $\hat{G}$ is torsion-free. 

\end{proof}

Now we are prepared to prove formula (2.1). 

\begin{theorem} Let $G$ be a torsion-free abelian group that is finitely generated with respect to its profinite topology. Then, for each $n\geq 0$, the map

$$\widehat{H_n(G)}\to H_n(\hat{G})$$

\noindent induced by $c_G:G\to \hat{G}$ is an isomorphism.
\end{theorem}

\begin{proof} By Lemma 2.8, $\hat{G}$ is torsion-free. In view of this, the result follows immediately from Theorem 2.1, Theorem 2.5 and Proposition 2.7.
\end{proof}

\section{Proof of the main theorem}

We begin by defining the class of groups that is the focus of this section.

\begin{definition} {\rm Define $\mathcal{G}$ to be the class of groups $G$ such that, for each $n\geq 0$ and each finite, discrete $\hat{G}$-module $A$, the following two properties hold: 
\vspace{5pt}

(i) the group $H^n(G,A)$ is finite;
\vspace{5pt}

(ii) the map $c_G:G\to \hat{G}$ induces an isomorphism $H^n(\hat{G},A)\to H^n(G,A)$.}
\end{definition}
\vspace{10pt}

Our objective is to prove that every ascending HNN extension of a polycyclic group is in $\mathcal{G}$.
In our proof we will make use of the fact that $\mathcal{G}$ is closed under the formation of the following type of group extension.

\begin{proposition}  Let $1\longrightarrow N\longrightarrow G\longrightarrow Q\longrightarrow 1$ be a group extension such that $N$ is finitely generated in its profinite topology. If $N$ and $Q$ are both in $\mathcal{G}$, then $G$ belongs to $\mathcal{G}$. 
 \end{proposition}

The essential details of the proof of the above proposition are provided by Serre [{\bf 14}, Exercise 2, Chapter 2], though with one important difference between the hypotheses: in place of our condition on $N$, Serre assumes that $N$ is finitely generated as an abstract group.  An examination of Serre's argument, however, reveals that all that is really required is that $N$ has only finitely many subgroups of any given finite index, a property that also holds in the presence of our weaker condition on $N$. 

In analysing ascending HNN extensions, the following species of direct limit shall play an important role.

\begin{definition}
{\rm If $G$ is a group and $\phi:G\to G$ an endomorphism, then $G_{\phi}$ is the direct limit of the sequence
$$G\stackrel{\phi}{\to} G\stackrel{\phi}{\to} G\stackrel{\phi}{\to} \cdots.$$}
\end{definition}
\vspace{10pt}

The above variety of direct limit enjoys the following property, which will be highly significant for our proof that every ascending HNN extension of a polycyclic group is in $\mathcal{G}$. 

\begin{lemma}
If $G$ is a finitely generated group and $\phi:G\to G$ an endomorphism, then $G_{\phi}$ is finitely generated relative to its profinite topology.  
\end{lemma}

\begin{proof} We have that $G_{\phi}$ is the direct limit of the sequence
\begin{equation}
G\stackrel{\phi}{\to} G\stackrel{\phi}{\to} G\stackrel{\phi}{\to} \cdots.
\end{equation}
Let $\epsilon : G_{\phi}\to F$ be an epimorphism, where $F$ is a finite group. For each $i\in \mathbb N$, 
let $N_i$ be the subgroup of $G$ formed by intersecting $\mbox{Ker}\ \epsilon$ with the copy of $G$ occupying the $i$-th spot in the sequence (3.1). We have, then, that $\phi(N_i)\leq N_{i+1}$ for all $i\in \mathbb N$, and that $\phi$ induces a monomorphism $G/N_i\to G/N_{i+1}$ for all $i\in \mathbb N$.  It follows from the finiteness of $F$ that there exists $k\in \mathbb N$ such that the map $\phi:G\to  G$ induces an isomorphism $G/N_i\to G/N_{i+1}$ for all $i\geq k$. Moreover, invoking the fact that $G$ has only finitely many subgroups of any given finite index, we can conclude that there is an $l\geq k$ such that $N_i=N_l$ for infinitely many $i\geq k$. To simplify the notation, we let $N=N_l$. 

For each nonnegative integer $j$, let
$$\phi^{-j}(N)=\{ x\in G : \phi^j(x)\in N \},$$
where $\phi^0$ is understood to be the identity map from $G$ to $G$.
Now set $M=\bigcap_{j=0}^{\infty}\phi^{-j}(N)$. It is easy to see that $M\leq N_i$ for all $i\in \mathbb N$,
$M\unlhd G$, and $\phi(M)\leq M$. We claim that, in addition, 
$[G:M]<\infty$ and the map $G/M\to G/M$ induced by $\phi$ is an isomorphism. To establish the former assertion, we first observe that, for each $j\geq 0$, $[G:\phi^{-j}(N)]\leq [G : N].$ Since 
$G$ has only finitely many subgroups with index $\leq [G:N]$, it follows that there are only finitely many subgroups of the 
form  $\phi^{-j}(N)$ for $j\geq 0$. Therefore, $M$, as the intersection of finitely many subgroups with finite index, has finite index. Turning now to  prove our second assertion about $M$,  we let $x\in G$ such that $\phi(x)\in M$. From the definition of $N$, we have that, for some $n>0$,  $\phi^n(N)\leq N$ and $\phi^n$ induces an isomorphism 
$G/N\to G/N$. 
Also, for any $j\geq 0$, $\phi^{n+j}(x)\in N$, implying that $\phi^j(x)\in N$. Hence $x\in M$. Therefore, the map $G/M\to G/M$ induced by $\phi$ is an isomorphism. 

Now let $\phi'$ be the map $M\to M$ induced by $\phi$. Treating $M_{\phi'}$ as a subgroup of $G_{\phi}$, we have $M_{\phi'}\leq \mbox{Ker}\ \epsilon$. Moreover, any element of $G_{\phi}$ is congruent modulo $M_{\phi'}$ to an element of the first $G$ in the sequence (3.1). It follows, then, that the image of the first copy of $G$ under $\epsilon$ is the entire group $F$. Consequently, we can conclude that $G_{\phi}$ is finitely generated relative to its profinite topology.  
\end{proof}

Below we establish the connection between the groups $G\ast_{\phi}$ and $G_{\phi}$.

\begin{lemma} If $G$ is a group and $\phi:G\to G$ an endomorphism, then   
$$G\ast_{\phi}\cong G_{\phi}\rtimes \mathbb Z.$$
\end{lemma}

\begin{proof} Each element of $G\ast_{\phi}$ can be written in the form $t^igt^{-j}$, where $i$ and $j$ are nonnegative integers and $g\in G$. Thus $G\ast_{\phi}$  is the product of the normal subgroup 
$\bigcup_{i=0}^{\infty} t^iGt^{-i}$ with the subgroup $\langle t\rangle$. Moreover, the commutative diagram

$$\begin{CD}
G@>\phi>> G@>\phi>> G@>\phi>> \cdots\\
@VV\theta_0V  @VV\theta_1V @V\theta_2VV\\
G@>\subset>> tGt^{-1}@>\subset>> t^2Gt^{-2}@>\subset>> \cdots,\\
\end{CD}$$
where $\theta_i(g)=t^igt^{-i}$, reveals that $G_{\phi}\cong \bigcup_{i=0}^{\infty} t^iGt^{-i}$. Hence the result follows.
\end{proof}

The above decomposition, combined with Lemma 3.2 and Proposition 3.1, yields the following corollary.

\begin{corollary}  Assume $G$ is a finitely generated group and $\phi:G\to G$ is an endomorphism. If $G_{\phi}$ is in $\mathcal{G}$, then $G\ast_{\phi}$ is also in $\mathcal{G}$.
\end{corollary}

Before proving our main theorem, we state a universal coefficient theorem for profinite groups. Rather than provide a proof, we refer the reader to the proof of the abstract version in [{\bf 15}, Theorem 3.6.5, Exercise 6.1.5], as it can easily be translated to the profinite context. 

\begin{theorem} Let $\Gamma$ be a profinite group and $A$ a trivial discrete $\Gamma$-module. Then, for any $n\geq 1$, there is an exact sequence

\begin{displaymath} \begin{CD}
0 @>>> {\rm Ext}(H_{n-1}(\Gamma), A) @>>> H^n(\Gamma,A) @>>> {\rm Hom}(H_n(\Gamma), A) @>>>
0\end{CD} \end{displaymath}
\end{theorem}

We now have everything in place to prove our principal result.

\begin{theorem} If $G$ is a polycyclic group and $\phi:G\to G$ an endomorphism,  then 
$G\ast_{\phi}$ is in the class $\mathcal{G}$.
\end{theorem} 

\begin{proof} By Corollary 3.4, it suffices to show that $G_{\phi}$ belongs to $\mathcal{G}$. We prove this assertion by induction on the length of the derived series of $G$.  First assume $G$ is abelian. Taking $A$ to be a finite, discrete $\widehat{G_{\phi}}$-module, we wish to establish the following two properties: 

(i) $H^n(G_{\phi},A)$ is finite;

(ii) $H^n(\widehat{G_{\phi}},A)\cong H^n(G_{\phi},A)$ for all $n\geq 0$. 

\noindent Before proving (i) and (ii) in general, we treat the special case where $G$ is torsion-free and the action of $G_{\phi}$ on $A$ is trivial. The two properties are clearly true for $n=0$, so we will assume that $n\geq 1$. In this case, the universal coefficient formulas yield the commutative diagram

\begin{equation} \begin{CD}
0 @>>> {\rm Ext}(H_{n-1}(\widehat{G_{\phi}}), A) @>>> H^n(\widehat{G_{\phi}},A) @>>> {\rm Hom}(H_n(\widehat{G_{\phi}}), A) @>>>0\\
&& @VVV @VVV @VVV && \\
0 @>>> {\rm Ext}(H_{n-1}(G_{\phi}), A) @>>> H^n(G_{\phi},A) @>>> {\rm Hom}(H_n(G_{\phi}), A) @>>>0
\end{CD} \end{equation}
with exact rows. According to Lemmas 2.8 and 3.2, $\widehat{G_{\phi}}$ must be torsion-free and topologically finitely generated. Hence $\widehat{G_{\phi}}$ is a direct sum of finitely many infinite procyclic groups, which means, by the K\"unneth formula, that the same property holds for $H_{n-1}(\widehat{G_{\phi}})$. As a result, ${\rm Ext}(H_{n-1}(\widehat{G_{\phi}}), A)=0$. 
Furthermore, $H_{n-1}(G)$ is torsion-free, which implies, since homology commutes with direct limits,  that $H_{n-1}(G_{\phi})$ is also torsion-free. Because $A$ is finite, this yields that
${\rm Ext}(H_{n-1}(G_{\phi}), A)=0$ (see Lemma 3.8 below). In addition, it follows from Theorem 2.9 that the third vertical map in (3.2) is an isomorphism. Therefore, property (ii) holds.  Also, since $H_n(\widehat{G_{\phi}})$ is a finitely generated profinite group, property (i) is true. 

Next we establish properties (i) and (ii) without the restriction that $A$ is a trivial $G_{\phi}$-module, still assuming, however, that $G$ is free abelian of finite rank. Let $\omega: G_{\phi}\to \mbox{Aut}(A)$ be the homomorphism arising from the action of $G_{\phi}$ on $A$. Arguing just as we did for the map $\epsilon$ in the proof of Lemma 3.2, we can find a subgroup $M$ in $G$ of finite index such that
$\phi(M)\leq M$ and $M_{\phi'}\leq \mbox{Ker}\ \omega$, where $\phi':M\to M$ is the map induced by $\phi$.  Now set $Q=G/M$, and let $\phi'': Q\to Q$ be the map induced by $\phi$. Then there is an exact sequence
\begin{equation} 1\to M_{\phi'}\to G_{\phi}\to Q_{\phi''}\to 1. \end{equation}
Noticing that $Q_{\phi''}$ is finite, we obtain from (3.3) an exact sequence
\begin{equation} 1\to \widehat{M_{\phi'}}\to \widehat{G_{\phi}}\to Q_{\phi''}\to 1 \end{equation}
of profinite groups. Moreover, by the case for a trivial module proved above, we have that $H^n(M_{\phi'},A)$ is finite and  $H^n\widehat{(M_{\phi'}},A)\cong H^n(M_{\phi'},A)$ for all $n\geq 0$.
Thus, invoking the Lyndon-Hochschild-Serre spectral sequences for  (3.3) and (3.4), we can conclude that both properties (i) and (ii) hold.  Therefore, $G_{\phi}$ lies in $\mathcal{G}$ whenever $G$ is free abelian of finite rank. 

We now treat the case where $G$ is a finitely generated, abelian group that may contain torsion. In this case, $G$ contains a torsion-free subgroup $N$ such that $G/N$ is finite. Taking $M=\bigcap_{j=0}^{\infty}\phi^{-j}(N)$, we have $M\leq N$, $[G:M]<\infty$, and $\phi(M)\leq M$. Hence, letting $Q$, $\phi'$ and $\phi''$ be exactly as in the previous paragraph, we have the exact sequences (3.3) and (3.4)  in this case, too. Also, by the torsion-free case proved above, $H^n(M_{\phi'},A)$ is finite and  $H^n\widehat{(M_{\phi'}},A)\cong H^n(M_{\phi'},A)$ for all $n\geq 0$. Properties (i) and (ii), then, follow as above.
Therefore, $G_{\phi}$ belongs to $\mathcal{G}$.  

Finally, we assume that the solvability length of $G$ exceeds $1$. Let $N$ be the commutator subgroup of $G$ and $Q=G/N$. Then $\phi(N)\leq N$. Let $\phi':N\to N$ and $\phi'': Q\to Q$ be the maps induced by $\phi$. Then we have an exact sequence
$$1\to N_{\phi'}\to G_{\phi}\to Q_{\phi''}\to 1.$$
By the base case, we have that $Q_{\phi''}$ is in $\mathcal{G}$, and, by the inductive hypothesis, $N_{\phi'}$ belongs to $\mathcal{G}$. Moreover, by Lemma 3.2, $N_{\phi'}$ is finitely generated with respect to its profinite topology. Therefore, by Proposition 3.1, $G_{\phi}$ belongs to $\mathcal{G}$. 
\end{proof}

A scrutiny of the above proof, particularly its third paragraph, reveals that the argument can be easily extended to prove that every ascending HNN extension of a virtually polycyclic group is in $\mathcal{G}$.

\begin{theorem} If $G$ is a virtually polycyclic group and $\phi:G\to G$ an endomorphism,  then 
$G\ast_{\phi}$ is in the class $\mathcal{G}$.
\end{theorem} 

It still remains to prove the following elementary result about abelian groups, which we invoked in the proof of Theorem 3.6.

\begin{lemma} If $A$ is a torsion-free abelian group, then ${\rm Ext}(A,B)=0$ for any finite abelian group $B$.
\end{lemma}

\begin{proof} We will prove the conclusion by showing that $\mbox{Ext}(A,\mathbb Z/p)=0$ for every prime $p$. Let $A_0=A\otimes \mathbb Q$, and consider the monomorphism $A\to A_0$. This map induces an epimorphism 
$\mbox{Ext}(A_0,\mathbb Z/p)\to \mbox{Ext}(A,\mathbb Z/p)$.
Hence the result will follow if we can establish that $\mbox{Ext}(A_0,\mathbb Z/p)=0$. To accomplish this, we employ the exact sequence
$$0\to \mathbb Z\stackrel{\times p}\to \mathbb Z\to \mathbb Z/p\to 0,$$
which gives rise to an exact sequence
$$\mbox{Ext}(A_0,\mathbb Z)\stackrel{\times p}\to \mbox{Ext}(A_0,\mathbb Z)\to \mbox{Ext}(A_0,\mathbb Z/p)\to 0.$$ Moreover, since multiplication by $p$ induces an isomorphism $A_0\to A_0$, the first map in the above sequence is an isomorphism, forcing the third group to be trivial.  

\end{proof}

In [{\bf 10}, Theorem 1.4] it is shown that every ascending HNN extension of a polycyclic group is a minimax group. In light of this, the question poses itself whether the result in Theorem 3.6 extends to all residually finite, solvable, minimax groups. Applying the same approach employed in the proof of Theorem 3.6, we provide a positive answer to this question in the following theorem. 

\begin{theorem} If $G$ is a residually finite, solvable, minimax group, then $G$ is in $\mathcal{G}$. \end{theorem}

\begin{proof} We begin by treating the case where $G$ is a residually finite, abelian, minimax group. 
As a minimax group, $G$ decomposes as $G=H\oplus T$, where $T$ is the torsion subgroup of $G$ and $H$ is torsion-free-- this is established in [{\bf 12}, Lemma 10.31(i)]. Moreover, since $G$ is residually finite, $T$ must be finite. According to [{\bf 12}, Lemma 10.31(ii)], $H$ is an extension of a free abelian group of finite rank by a direct sum of finitely many quasicyclic groups. Since both kernel and quotient in this extension are finitely generated relative to their profinite topologies, it follows that $H$, too, enjoys this property. Thus, invoking the same argument that was employed in the first two paragraphs of the proof of Theorem 3.6 in order to establish that $G_{\phi}$ is in $\mathcal{G}$, we can deduce that $H$ must belong to the class $\mathcal{G}$. Hence $G$ resides in this class as well.  

Now we consider the general case. Since $G$ is residually finite, it possesses an abelian normal series consisting solely of subgroups that are closed in the profinite topology of $G$. We will prove that $G$ is in $\mathcal{G}$ by inducting on the length of this series, the case of a series of length one having been disposed of above. Let $A$ be the first subgroup in the series. Then $A$ is in $\mathcal{G}$ by our base case. Since $A$ is closed, $G/A$ is residually finite, which means, in view of the inductive hypothesis, that it belongs to $\mathcal{G}$. In addition, by our reasoning in the first paragraph, $A$ is finitely generated with regard to its profinite topology. Therefore, Proposition 3.1 yields that $G$ is in $\mathcal{G}$. 
\end{proof}

\begin{acknowledgement}{\rm The author benefited from discussions with Pavel Zalesskii about profinite completions of ascending HNN extensions. In addition, he is indebted to Peter Symonds for enlightening him regarding profinite homology.}
\end{acknowledgement}

\end{document}